\documentclass[final,3p]{elsarticle}
\usepackage[latin1]{inputenc}
 \usepackage{graphics}
 \usepackage{graphicx}
 \usepackage{epsfig}
\usepackage{amssymb}
 \usepackage{amsthm}
 \usepackage{lineno}
 \usepackage{amsmath}
\usepackage{color}
   \numberwithin{equation}{section}
\usepackage{mathrsfs}

\journal{``Journal of Pseudo-Differential Operators and Applications"} 

\newtheorem{thm}{Theorem}[section]

\newtheorem{lem}[thm]{Lemma}

\newtheorem{defn}[thm]{Definition}

\biboptions{sort&compress,square}
\allowdisplaybreaks
\begin{document}
\begin{frontmatter}
\author{Tong Wu$^{a}$}
\ead{wut977@nenu.edu.cn}
\author{Yong Wang$^{b,*}$}
\ead{wangy581@nenu.edu.cn}
\cortext[cor]{Corresponding author.}
\address{$^a$Department of Mathematics, Northeastern University, Shenyang, 110819, China}
\address{$^b$School of Mathematics and Statistics, Northeast Normal University,
Changchun, 130024, China}

\title{The spectral torsion for the rescaled Dirac operator}
\begin{abstract}
 In the paper, we give four different examples of the rescaled Dirac operator by the perturbation of the function $f$. Further, based on the trilinear Clifford multiplication by functional of differential one-forms, we compute the spectral torsion for four kinds of rescaled Dirac operator on even-dimensional oriented compact spin Riemannian manifolds without boundary.
\end{abstract}
\end{frontmatter}
\section{Introduction}
 Until now, many geometers have studied the noncommutative residues. In \cite{Gu,Wo}, authors found the noncommutative residues are of great importance to the study of noncommutative geometry. Wodzicki \cite{Wo} first introduced the concept of the noncommutative residue in the study of higher-dimensional manifolds, namely: the noncommative residue is a trace over the algebra of all classical pseudodifferential operators on a closed compact manifold. However, this trace is not an extension of the usual trace. The noncommutative residue can also be called the Wodzicki residue. Let $\Phi:\Sigma\rightarrow R^d$ be Riemannian surface, where $\Phi=(\phi_1,\cdot\cdot\cdot,\phi_d)$ is a smooth embedding, $g=\sum_{ij}\eta_{ij}d\phi^i\otimes d\phi^j$ is the metric on Riemannian surface. Then Polyakov action is defined by
 $I:=\frac{1}{2\pi}\int_M\eta_{ij}d\phi^i\otimes d\phi^j.$ In \cite{Co1}, Connes used the noncommutative residue to derive a conformal 4-dimensional Polyakov action analogy.  Connes showed us that the noncommutative residue on a compact manifold $M$ coincided with the Dixmier's trace(see $\S7.5$ in \cite{J1}) on pseudo-differential operators of order $-{\rm {dim}}M$ in \cite{Co2}.
And Connes claimed the noncommutative residue of the square of the inverse of the Dirac operator was proportioned to the Einstein-Hilbert action. That is $${\rm Wres}(D^{-2})=c_{0}\int_{M}sdvol_{M},$$
where $c_0$ is a constant and $s$ is the scalar curvature.  Kastler \cite{Ka} gave a
brute-force proof of this theorem. Kalau and Walze proved this theorem in the normal coordinates system simultaneously in \cite{KW}. Therefore, we call it the Kastler-Kalau-Walze theorem. Ackermann proved that
the Wodzicki residue of the square of the inverse of the Dirac operator ${\rm  Wres}(D^{-2})$ in turn is essentially the second coefficient
of the heat kernel expansion of $D^{2}$ in \cite{Ac}.

 Recently, the significance of the spectral torsion has been emphasized in a somewhat distinct context. In
 \cite{DL3}, Dabrowski, Sitarz and Zaleck proposed a plain, purely spectral method that allowed to determine the torsion as
 the density of the torsion functional and imposed the torsion-free condition for regular finitely summable
 spectral triples, which is a first step towards linking the spectral approach with the algebraic approach
 based on Levi-Civita connections. Dirac operators with torsion are by now well-established analytical tools
 in the study of special geometric structures. Ackermann and Tolksdorf \cite{Ac1} proved a generalized version of
 the well-known Lichnerowicz formula for the square of the most general Dirac operator with torsion $D_T$ on
 an even-dimensional spin manifold associated to a metric connection with torsion. In \cite{PF1,PF2}, Pf$\ddot{a}$ffle and
 Stephan considered orthogonal connections with arbitrary torsion on compact Riemannian manifolds, and for
 the induced Dirac operators, twisted Dirac operators and Dirac operators of Chamseddine-Connes type they
 computed the spectral action. Sitarz and Zajac \cite{sa} investigated the spectral action for scalar perturbations
 of Dirac operators. Wang \cite{Wa0} considered the arbitrary perturbations of Dirac operators, and established the
 associated Kastler- Kalau-Walze theorem. In \cite{WWY}, Wang, Wang and Yang gave two kinds of operator-theoretic
explanations of the gravitational action about Dirac operators with torsion in the case of 4-dimensional compact manifolds with
flat boundary.  In \cite{WWw}, we gave some new
spectral functionals which is the extension of spectral functionals to the noncommutative realm with torsion, and related them to the
noncommutative residue for manifolds with boundary about Dirac operators with torsion.

 Based on the spectral torsion and the noncommutative residue, Dabrowski et al. \cite{DL3} showed that the
 spectral definition of torsion can be readily extended to the noncommutative case of spectral triples. By
 twisting the spectral triple of a Riemannian spin manifold, Martinetti et al. showed how to generate
 an orthogonal and geodesic preserving torsion from a torsionless Dirac operator in [28]. Hong and Wang computed the spectral Einstein functional associated with the Dirac operator with torsion
 on even-dimensional spin manifolds without boundary in \cite{Hj}. In \cite{WWj}, Wang and Wang provide an explicit computation of the spectral torsion
 associated with the Connes type operator on even dimension compact manifolds.
In \cite{Si1}, Sitarz proposed a new idea of conformally rescaled and curved spectral triples, which
are obtained from a real spectral triple by a nontrivial scaling of the Dirac
operator. In \cite{wt}, we compute the noncommutative residue for the rescaled Dirac operator $fDh$ on 6-dimensional
compact manifolds without boundary. And we also give some
important special cases which can be solved by our calculation methods.
 {\bf The motivation} of this paper is to compute the spectral torsion for the rescaled Dirac operator with the trilinear Clifford multiplication by functional of differential one-forms $c(u),c(v),c(w)$ on even-dimensional oriented compact spin Riemannian manifolds without boundary, where $c(u)=\sum_{r=1}^{n} u_{r}c(e_r), c(v)=\sum_{p=1}^{n} v_{p}c(e_{p}), c(w)=\sum_{q=1}^{n} w_{q}c(e_{q})$.

This paper is organized as follows. In Section \ref{section:2}, we introduce some notations about Clifford action and the rescaled Dirac operator. Using the residue for a differential operator of Laplace type and the composition formula of pseudo-differential operators, some general symbols of the generalized laplacian for the rescaled Dirac operator are given in Section \ref{section:3}. We compute the spectral torsion for four kinds of rescaled Dirac operator on even-dimensional oriented compact spin Riemannian manifolds without boundary in Section \ref{section:4} and \ref{section:5}.
\section{ Preliminaries for the rescaled Dirac operator}
\label{section:2}
 The purpose of this section is to introduce some notations about Clifford action and the rescaled Dirac operator.

Let $M$ be an $=2m$-dimensional ($n\geq 3$) oriented compact spin Riemannian manifold with a Riemannian metric $g$. And let $\nabla^L$ be the Levi-Civita connection about $g$. In the
fixed orthonormal frame $\{e_1,\cdots,e_n\}$ in $TM$, $TM$ (resp. $T^*M$) denote the tangent (resp. cotangent) vector bundle of $M$, the connection matrix $(\omega_{s,t})$ is defined by
\begin{equation}
\label{eq1}
\nabla^L(e_1,\cdots,e_n)= (e_1,\cdots,e_n)(\omega_{s,t}).
\end{equation}

 Let $c(e), \hat{c}(e)$ be the Clifford operators acting on the exterior algebra bundle $\Lambda^*(T^*M)$ of $T^*M$ defined by
 $$c(e)=e^*\wedge-i_e,~~~\hat{c}(e)=e^*\wedge+i_e,$$
where $e^*$ and $i_e$ are the notation for the exterior and interior multiplications respectively.
For  $\{e_1,\cdots,e_n\}$, one has
\begin{align}
\label{ali1}
&\hat{c}(e_i)\hat{c}(e_j)+\hat{c}(e_j)\hat{c}(e_i)=2g(e_i,e_j)=2\delta^i_j;~~\nonumber\\
&c(e_i)c(e_j)+c(e_j)c(e_i)=-2g(e_i,e_j)=-2\delta^i_j;~~\nonumber\\
&c(e_i)\hat{c}(e_j)+\hat{c}(e_j)c(e_i)=0.
\end{align}
By \cite{Wa3}, we have the Dirac operator
\begin{align}
D&=\sum^n_{i=1}c(e_i)\bigg(e_i-\frac{1}{4}\sum_{s,t}\omega_{s,t}
(e_i)c(e_s)c(e_t)\bigg).
\end{align}
The symbol expansion of a parametrix of $D$ is given,
\begin{align}\label{sigma0}
\sigma_{1}(D)=\sqrt{-1}c(\xi);~~~
		\sigma_{0}(D)=-\frac{1}{4}\sum_{i,s,t}\omega_{st}(e_i)c(e_i)c(e_s)c(e_t).
	\end{align}
Consider the rescaled Dirac operator, which is defined as
 \begin{align}\label{xx}
f(D+\mathbb{A})f=f\bigg[\sum^n_{i=1}c(e_i)\bigg(e_i-\frac{1}{4}\sum_{s,t}\omega_{s,t}
(e_i)c(e_s)c(e_t)\bigg)+\mathbb{A} \bigg]f,
\end{align}
where $f$ is smooth function on $M$, $f(x)\neq 0,$ for $\forall x \in M$ and $\mathbb{A}$ denotes the Clifford multiplication by any form.\\
Then we obtain the leading symbols of $f(D+\mathbb{A})f$:
	\begin{align}\label{sigma0}
\sigma_{1}(f(D+\mathbb{A})f)=\sqrt{-1}fc(\xi)f;~~~
		\sigma_{0}(f(D+\mathbb{A})f)=-\frac{1}{4}f\sum_{i,s,t}\omega_{st}(e_i)c(e_i)c(e_s)c(e_t)f+f\mathbb{A} f.
	\end{align}

Next, we want to get the leading symbols of $(f(D+\mathbb{A})f)^2$.
For a differential operator of Laplace type $P$, it has locally the form
\begin{equation}\label{p3}
	P=-(g^{ij}\partial_i\partial_j+A^i\partial_i+B),
\end{equation}
where $\partial_{i}$  is a natural local frame on $TM,$ $(g^{ij})_{1\leq i,j\leq n}$ is the inverse matrix associated to the metric
matrix  $(g_{ij})_{1\leq i,j\leq n}$ on $M,$ $A^{i}$ and $B$ are smooth sections of $\textrm{End}(N)$ on $M$ (endomorphism).\\
\indent Write the Dirac operators $D^2$ and $D^{-1}$ by different orders as:
 \begin{eqnarray}
D_x^{\alpha}&=(-i)^{|\alpha|}\partial_x^{\alpha};
~\sigma(D^2)=p_2+p_1+p_0;
~\sigma(D^{-1})=\sum^{\infty}_{j=1}q_{-j}.
\end{eqnarray}

\indent By the composition formula of pseudo-differential operators, we have
\begin{align}
1=\sigma(D^2\circ {D}^{-2})&=\sum_{\alpha}\frac{1}{\alpha!}\partial^{\alpha}_{\xi}[\sigma({D^2})]
{D}_x^{\alpha}[\sigma({D}^{-2})]\nonumber\\
&=(p_2+p_1+p_0)(q_{-2}+q_{-3}+q_{-4}+\cdots)\nonumber\\
&~~~+\sum_j(\partial_{\xi_j}p_2+\partial_{\xi_j}p_1+\partial_{\xi_j}p_0)(
D_{x_j}q_{-2}+D_{x_j}q_{-3}+D_{x_j}q_{-4}+\cdots)\nonumber\\
&~~~+\sum_{i,j}(\partial_{\xi_i}\partial_{\xi_j}p_2+\partial_{\xi_i}\partial_{\xi_j}p_1+\partial_{\xi_i}\partial_{\xi_j}p_0)(
D_{x_i}D_{x_j}q_{-2}+D_{x_i}D_{x_j}q_{-3}+D_{x_i}D_{x_j}q_{-4}+\cdots)\nonumber\\
&=p_2q_{-2}+(p_1q_{-2}+p_2q_{-3}+\sum_j\partial_{\xi_j}p_2D_{x_j}q_{-2})+(p_0q_{-2}+p_1q_{-3}+p_2q_{-4}\nonumber\\
&+\sum_j\partial_{\xi_j}p_2D_{x_j}q_{-3}+\sum_{i,j}\partial_{\xi_i}\partial_{\xi_j}p_2D_{x_i}D_{x_j}q_{-2})+\cdots,
\end{align}
so
\begin{equation}
q_{-2}=p_2^{-1};~q_{-3}=-p_2^{-1}[p_1p_1^{-2}+\sum_j\partial_{\xi_j}p_2D_{x_j}q_{-2}].
\end{equation}

To get the leading symbols of $(f(D+\triangle)f)^2$, we first expand it.
\begin{align}\label{p4}
(f(D+\mathbb{A})f)^2&=f(D+\mathbb{A})f^2(D+\mathbb{A})f\nonumber\\
&=f^4(D+\mathbb{A})^2+fc(df^3)(D+\mathbb{A})+f^3(D+\mathbb{A})c(df)+fc(df^2)c(df).
\end{align}
Obviously, we only need to further expand $(D+\mathbb{A})^2$.
\begin{align}
(D+\mathbb{A})^2=D^2+D\mathbb{A}+\mathbb{A} D+\mathbb{A}^2.
\end{align}
\indent Let $g^{ij}=g(dx_{i},dx_{j})$, $\xi=\sum_{j}\xi_{j}dx_{j}$ and $\nabla^L_{\partial_{i}}\partial_{j}=\sum_{k}\Gamma_{ij}^{k}\partial_{k}$,  we denote that
\begin{align}\label{p2}
&\sigma_{i}=-\frac{1}{4}\sum_{s,t}\omega_{s,t}
(e_i)c(e_s)c(e_t)
;\nonumber\\
&\xi^{j}=g^{ij}\xi_{i};~~~~\Gamma^{k}=g^{ij}\Gamma_{ij}^{k};~~~~\sigma^{j}=g^{ij}\sigma_{i}.
\end{align}
Then by (\ref{p3}) and (\ref{p2}), we have
\begin{align}
D\mathbb{A}+\mathbb{A} D=\sum_{i,j}g^{ij}\bigg(c(\partial_i)\mathbb{A}+\mathbb{A} c(\partial_i)\bigg)\partial_j+\sum_{i,j}g^{ij}\bigg(c(\partial_i)\partial_j(\mathbb{A})+c(\partial_i)\sigma_j\mathbb{A}+\mathbb{A} c(\partial_i)\sigma_j\bigg).
\end{align}
Further, the following result is obtained.
\begin{align}
(D+\mathbb{A})^2&=-\sum_{i,j}g^{ij}\partial_i\partial_j+\bigg(-2\sigma^j+\Gamma^j+c(\partial_i)\mathbb{A}+\mathbb{A} c(\partial_i)\bigg)\partial_j+\sum_{i,j}g^{ij}\bigg(-\partial_i(\sigma_j)-\sigma_i\sigma_j\nonumber\\
&+\Gamma^k_{ij}\sigma_k+c(\partial_i)\partial_j(\mathbb{A})+c(\partial_i)\sigma_j\mathbb{A}+\mathbb{A} c(\partial_i)\sigma_j\bigg)+\frac{1}{4}s+\mathbb{A}^2,
\end{align}
where $s$ denotes the scalar curvature.

Using (\ref{p4}), the leading symbols of $(f(D+\mathbb{A})f)^2$ are given.
\begin{lem}\label{lem331}
	The following identities hold:
	\begin{align}\label{sigma0}
\sigma_{2}[(f(D+\mathbb{A})f)^2](x,\xi)&= f^4\|\xi\|^2;\nonumber\\
		\sigma_{1}[(f(D+\mathbb{A})f)^2](x,\xi)&=\sqrt{-1}f^4\bigg(\Gamma^j-2\sigma^j+c(\partial^j)\mathbb{A}+\mathbb{A} c(\partial^j)\bigg)\xi_j+\sqrt{-1}fc(df^3)c(\xi)+\sqrt{-1}f^3c(\xi)c(df);\nonumber\\
\sigma_{0}[(f(D+\mathbb{A})f)^2](x,\xi)&=f^4\bigg\{g^{ij}\bigg(-\partial_i(\sigma_j)-\sigma_i\sigma_j+\Gamma^k_{ij}\sigma_k+c(\partial_i)\partial_j(\mathbb{A})+c(\partial_i)\sigma_j\mathbb{A} +\mathbb{A} c(\partial_i)\sigma_j\bigg)+\frac{1}{4}s\nonumber\\
&+\mathbb{A} ^2\bigg\}+fc(df^3)\mathbb{A} +f^3\mathbb{A} c(df)+fc(df^2)c(df).
	\end{align}
\end{lem}

 \section{Trilinear functional for the rescaled Dirac operator $f(D+\mathbb{A})f$}
 \label{section:3}

In this section, we consider the trilinear functional for the rescaled
 Dirac operator $f(D+\mathbb{A})f$. For our purpose, we provide some basic results through for later calculations.
\begin{defn}\cite{DL3}
Let $c(u)=\sum_{r=1}^{n} u_{r}c(e_r), c(v)=\sum_{p=1}^{n} v_{p}c(e_{p}), c(w)=\sum_{q=1}^{n} w_{q}c(e_{q}),$ for $f(D+\mathbb{A}) f$ given by (\ref{xx}), the trilinear Clifford multiplication by functional of differential one-forms $c(u), c(v), c(w)$
\begin{align}
 	\mathscr{S}_{f(D+\triangle) f}\bigg(c(u),c(v),c(w)\bigg)&=\mathrm{Wres}\bigg(c(u)c(v)c(w)(f(D+\mathbb{A}) f)^{-2m+1}\bigg)
 \end{align}
  is called torsion functional about the rescaled Dirac operator $f(D+\mathbb{A})f$.
\end{defn}
For a pseudo-differential operator  $P$, acting on sections of a vector bundle over an even dimensional compact Riemannian manifold $M$, the analogue of the volume element in the noncommutative geometry is the operator  $D^{-n}:= d s^{n} $. And pertinent operators are realized as pseudo-differential operators on the spaces of sections. Extending previous definitions by Connes \cite{co5}, a noncommutative integral was introduced in \cite{FGV2} based on the noncommutative residue \cite{wo2}, combine (1.4) in \cite{co4} and \cite{Ka}, using the definition of the residue:
\begin{align}\label{wers}
	\int P d s^{n}:=\operatorname{Wres} (P D^{-n}):=\int_{M}\int_{\|\xi\|=1} \operatorname{tr}\left[\sigma_{-n}\left(P D^{-n}\right)\right](x, \xi)\sigma(\xi)dx,
\end{align}
where  $\sigma_{-n}\left(P D^{-n}\right) $ denotes the  $(-n)$th order piece of the complete symbols of  $P D^{-n} $,  $\operatorname{tr}$  as shorthand of trace.

Firstly, we review here technical tool of the computation, which are the integrals of polynomial functions over the unit spheres. By (32) in \cite{B1}, we define
\begin{align*}
I_{S_n}^{\gamma_1\cdot\cdot\cdot\gamma_{2\bar{n}+2}}=\int_{|x|=1}d^nxx^{\gamma_1}\cdot\cdot\cdot x^{\gamma_{2\bar{n}+2}},
\end{align*}
i.e. the monomial integrals over a unit sphere.
Then by Proposition A.2. in \cite{B1}, polynomial integrals over higher spheres in the $n$-dimesional case are given
\begin{align}
I_{S_n}^{\gamma_1\cdot\cdot\cdot\gamma_{2\bar{n}+2}}=\frac{1}{2\bar{n}+n}[\delta^{\gamma_1\gamma_2}I_{S_n}^{\gamma_3\cdot\cdot\cdot\gamma_{2\bar{n}+2}}+\cdot\cdot\cdot+\delta^{\gamma_1\gamma_{2\bar{n}+1}}I_{S_n}^{\gamma_2\cdot\cdot\cdot\gamma_{2\bar{n}+1}}],
\end{align}
where $S_n\equiv S^{n-1}$ in $\mathbb{R}^n$.

For $\bar{n}=0$, we have $I^0={\rm Vol}(S^{n-1})$=$\frac{2\pi^{\frac{n}{2}}}{\Gamma(\frac{n}{2})}$, and we immediately get
\begin{align}
I_{S_n}^{\gamma_1\gamma_2}&=\frac{1}{n}{\rm Vol}(S^{n-1})\delta^{\gamma_1}_{\gamma_2}.
\end{align}

By \eqref{wers}, to obtain the results of the above torsion functional about the rescaled Dirac operator $f(D+\mathbb{A})f$, we need to compute
\begin{align}\label{abd}
\int_{M}\int_{\|\xi\|=1} \operatorname{tr}\left[\sigma_{-2 m}\bigg(c(u)c(v)c(w)(f(D+\mathbb{A})f)^{-2m+1}\bigg)\right](x, \xi)\sigma(\xi)dx.
\end{align}

Based on the algorithm yielding the principal symbol of a product of pseudo-differential operators in terms of the principal symbols of the factors, and by lemma \ref{lem331}, we get
\begin{align}\label{ABD}
&	\sigma_{-2 m}\left(c(u)c(v)c(w)(f(D+\mathbb{A})f)^{-2 m+1}\right)\nonumber\\
&=c(u)c(v)c(w)\sigma_{-2 m}\left((f(D+\mathbb{A})f)^{-2 m}\cdot(f(D+\mathbb{A})f)\right) \nonumber\\
& =c(u)c(v)c(w)\left\{\sum_{|\alpha|=0}^{\infty} \frac{(-i)^{|\alpha|}}{\alpha!} \partial_{\xi}^{\alpha}[\sigma((f(D+\mathbb{A})f)^{-2 m})] \partial_{x}^{\alpha}\left[\sigma\left(f(D+\mathbb{A})f\right)\right]\right\}_{-2 m} \nonumber\\
	& =c(u)c(v)c(w)\sigma_{-2 m}\left((f(D+\mathbb{A})f)^{-2 m}\right)\sigma_{0}(f(D+\mathbb{A})f)\nonumber\\
 &+c(u)c(v)c(w)\sigma_{-2 m-1}\left((f(D+\mathbb{A})f)^{-2 m}\right)\sigma_{1}(f(D+\mathbb{A})f)\nonumber\\
	& +c(u)c(v)c(w)(-\sqrt{-1}) \sum_{j=1}^{2m} \partial_{\xi_{j}}\left[\sigma_{-2 m}\left((f(D+\mathbb{A})f)^{-2 m}\right)\right] \partial_{x_{j}}\left[\sigma_{1}(f(D+\mathbb{A})f)\right].
\end{align}
Write $\sigma_2^{(-m+1)}:=[\sigma_{-2}((f(D+\mathbb{A})f)^{-2})]^{m-1}$, then by (3.8) in \cite{WJ}, we have
\begin{align}\label{9000}
\sigma_{-2m-1}[(f(D+\mathbb{A})f)^{-2m}]=m\sigma_2^{(-m+1)}\sigma_{-3}[(f(D+\mathbb{A})f)^{-2}]-\sqrt{-1}\sum_{k=0}^{m-2}\sum_{\mu=1}^{2m}\partial_{\xi_\mu}\sigma_2^{(-m+k+1)}\partial_{x_\mu}\sigma_2^{-1}(\sigma_2^{-1})^k.
\end{align}
By lemma \ref{lem331} and the composition formula of pseudo-differential operators, the following results are given.
\begin{align}
&\sigma_2^{(-m+1)}=f^{-4m+4}\|\xi\|^{-2m+2};~~~(\sigma_2^{-1})^k=f^{-4k}\|\xi\|^{-2k};\nonumber\\
&\partial_{\xi_\mu}\sigma_2^{(-m+k+1)}=2(-m+k+1)f^{-4m+4k+4}\|\xi\|^{-2m+2k}\xi^\mu;\nonumber\\
&\partial_{x_\mu}\sigma_2^{-1}=\partial_{x_\mu}(f^{-4})\|\xi\|^{-2}-f^{-4}\|\xi\|^{-4}\xi_\alpha\xi_\beta\partial_{x_\mu}g^{\alpha\beta},
\end{align}
and
\begin{align}
\sigma_{-3}[(f(D+\mathbb{A})f)^{-2}]&=-\sqrt{-1}f^{-4}\|\xi\|^{-4}\bigg(\Gamma^\mu-2\sigma^\mu+c(\partial^\mu)\mathbb{A}+\mathbb{A} c(\partial^\mu)\bigg)\xi_\mu-\sqrt{-1}f^{-7}\|\xi\|^{-4}c(df^3)c(\xi)\nonumber\\
&-\sqrt{-1}f^{-5}\|\xi\|^{-4}c(\xi)c(df)+2\sqrt{-1}\|\xi\|^{-4}\xi^\mu\partial_{x_\mu}(f^{-4})-2\sqrt{-1}f^{-4}\|\xi\|^{-6}\xi^\mu\xi_\alpha\xi_\beta\partial_{x_\mu}g^{\alpha\beta}.
\end{align}

Further, substituting above results into (\ref{9000}), we obtain
	\begin{align}\label{sigma}
		\sigma_{-2m-1}[(f(D+\mathbb{A})f)^{-2m}]&=-\sqrt{-1}mf^{-4m}\|\xi\|^{-2m-2}\bigg(\Gamma^\mu-2\sigma^\mu+c(\partial^\mu)\mathbb{A}+\mathbb{A} c(\partial^\mu)\bigg)\xi_\mu\nonumber\\
&-\sqrt{-1}mf^{-4m-3}\|\xi\|^{-2m-2}c(df^3)c(\xi)-\sqrt{-1}mf^{-4m-1}\|\xi\|^{-2m-2}c(\xi)c(df)\nonumber\\
&+2\sqrt{-1}mf^{-4m+4}\|\xi\|^{-2m-2}\xi^\mu\partial_{x_\mu}(f^{-4})-2\sqrt{-1}mf^{-4m}\|\xi\|^{-2m-4}\xi^\mu\xi_\alpha\xi_\beta\partial_{x_\mu}g^{\alpha\beta}\nonumber\\
&-2\sqrt{-1}\sum_{k=0}^{m-2}\sum_{\mu=1}^{2m}f^{-4m+4}(-m+k+1)\|\xi\|^{-2m-2}\xi^\mu\partial_{x_\mu}(f^{-4})\nonumber\\
&+2\sqrt{-1}\sum_{k=0}^{m-2}\sum_{\mu=1}^{2m}f^{-4m}(-m+k+1)\|\xi\|^{-2m-4}\xi^\mu\xi_\alpha\xi_\beta\partial_{x_\mu}g^{\alpha\beta}.
	\end{align}
And by lemma \ref{lem331}, in normal coordinates, we can get each item in (\ref{ABD}).
 \section{The spectral torsion for the rescaled Dirac operator $f(D+\mathbb{A})f$}
 \label{section:4}
  In this section, we develop the several examples of the rescaled Dirac operator $f(D+\mathbb{A})f$, and compute the spectral torsion for them respectively.
\subsection{The spectral torsion for the rescaled Dirac operator $f(D+c(T))f$}
Defining the 3-form $T$, we obtain a new covariant derivative
\begin{equation}
		\langle \nabla_{X}^{T} Y, Z \rangle=\langle \nabla_{X}^{L} Y, Z \rangle + T(X, Y, Z),\nonumber
\end{equation}
where $\nabla^{T}$ denotes the metric connection.

Lift $\nabla^T$ to $\nabla^{S(TM),T}$ on $S(TM)$, let $\mathbb{A}=c(T)=\frac{3}{2}\sum_{1\leqslant j< l< t\leqslant n}T(e_j, e_l, e_t)c(e_{j})c(e_{l})c(e_{t})$, then the Dirac operator with torsion $D_T$ is defined as:
\begin{align}\label{p1}
	D_T&=\sum_{j=1}^{n}c(e_{j})\nabla_{e_{j}}^{S(TM),T}\nonumber\\
		&=\sum_{j=1}^{n}c(e_{j})\bigg(e_{j}+\frac{1}{4}\sum_{l, t=1}^{n}\langle \nabla_{e_{j}}^{T}e_{l}, e_{t}\rangle c(e_{l})c(e_{t})\bigg)\nonumber\\
		&=D+\frac{1}{4}\sum_{j,l,t=1}^{n}T(e_j, e_l, e_t)c(e_{j})c(e_{l})c(e_{t})\nonumber\\
		&=D+\frac{3}{2}\sum_{1\leqslant j< l< t\leqslant n}T(e_j, e_l, e_t)c(e_{j})c(e_{l})c(e_{t})\nonumber\\
&=D+c(T).
\end{align}
By the relation of the Clifford action and ${\rm tr}(AB)={\rm tr}(BA)$, we have the following lemma.
\begin{lem}\label{pppp}
The following identities hold:
\begin{align}
&(1)~~~~{\rm tr}\bigg(c(u)c(v)c(w)c(T)\bigg)=\frac{3}{2}T(u,v,w){\rm tr}[id];\nonumber\\
&(2)~~~~\operatorname{tr} \biggl(c(u)c(v)c(w)c(df^3)\biggr)=\bigg(g(v,w)u(f^3)-g(u,w)v(f^3)+g(u,v)w(f^3)\bigg)
{\rm tr}[id].
\end{align}
\end{lem}
\begin{proof}
\begin{align}\label{t12}
	(1)~~~~~&{\rm tr}\bigg(c(u)c(v)c(w)c(T)\bigg)\nonumber\\
	&=\frac{3}{2}\sum_{1\leq j<l<t\leq n}T(e_j,e_l,e_t){\rm tr}\bigg(c(u)c(v)c(w)c(e_j)c(e_l)c(e_t)\bigg)\nonumber\\
&=\frac{3}{2}\sum_{\substack{1\leq j<l<t\leq n\\r,p,q}}T(e_j,e_l,e_t)u_rv_pw_q{\rm tr}\bigg(c(e_r)c(e_p)c(e_q)c(e_j)c(e_l)c(e_t)\bigg)\nonumber\\
&=\frac{3}{2}\sum_{\substack{1\leq j<l<t\leq n\\r,p,q}}T(e_j,e_l,e_t)u_rv_pw_q\bigg(\delta_t^q(\delta_l^p\delta_j^r-\delta_l^r\delta_j^p)-\delta_t^p(\delta_l^q\delta_j^r-\delta_l^r\delta_j^q)+\delta_t^r(\delta_l^q\delta_j^p-\delta_l^p\delta_j^q)\bigg){\rm tr}[id]\nonumber\\
&=\frac{3}{2}\sum_{1\leq j<l<t\leq n}T(e_j,e_l,e_t)\bigg(u_jv_lw_t-u_lv_jw_t-u_jv_tw_l+u_lv_tw_j+u_tv_jw_l-u_tv_lw_j\bigg){\rm tr}[id]\nonumber\\
&=\frac{3}{2}\bigg(\sum_{1\leq j<l<t\leq n}T(e_j,e_l,e_t)+\sum_{1\leq l<j<t\leq n}T(e_j,e_l,e_t)+\sum_{1\leq j<t<l\leq n}T(e_j,e_l,e_t)\nonumber\\
&+\sum_{1\leq t<j<l\leq n}T(e_j,e_l,e_t)+\sum_{1\leq l<t<j\leq n}T(e_j,e_l,e_t)+\sum_{1\leq t<l<j\leq n}T(e_j,e_l,e_t)\bigg)u_jv_lw_t{\rm tr}[id]\nonumber\\
&=\frac{3}{2}T(u,v,w){\rm tr}[id];\\
(2)~~~~~&\operatorname{tr} \biggl(c(u)c(v)c(w)c(df^3)\biggr)\nonumber\\
&=\sum_{\mu,r,p,q}\bigg(\partial_{x_\mu}(f^3)u_rv_pw_q(\delta^r_p\delta_q^\mu-\delta_r^q\delta_\mu^p+\delta_r^\mu\delta_p^q)\bigg){\rm tr}[id]\nonumber\\
&=\sum_{\mu,r,p,q}\bigg(\partial_{x_q}(f^3)u_rv_rw_q-\partial_{x_p}(f^3)u_qv_pw_q+\partial_{x_r}(f^3)u_rv_pw_p\bigg){\rm tr}[id]\nonumber\\
&=\bigg(g(v,w)u(f^3)-g(u,w)v(f^3)+g(u,v)w(f^3)\bigg)
{\rm tr}[id].
\end{align}
\end{proof}
 For any fixed point $x_0\in M$, we can choose the normal coordinates $U$ of $x_0$ in $M$. Then we have the following lemma.
\begin{lem}
In the normal coordinates $U$ of $x_0$ in $M$,
\begin{align}
\sum_{i,s,t}w_{st}(e_i)c(e_i)c(e_s)c(e_t)(x_0)=0;~~~\Gamma^k(x_0)=0;~~~\sigma^k(x_0)=0;~~~\partial_{x_k}g^{\alpha\beta}(x_0)=0.
\end{align}
\end{lem}
Next we arrive to our first main result.
 \begin{thm}\label{thm}
 	Let $M$ be an $n=2m$ dimensional ($n\geq 3$) oriented compact spin Riemannian manifold, for the rescaled Dirac operator with the trilinear Clifford multiplication by functional of differential one-forms $c(u),c(v),c(w),$ the spectral torsion for $f(D+c(T))f$ equals to
 \begin{align}
 	&\mathscr{S}_{f(D+c(T))f}\bigg(c(u),c(v),c(w)\bigg)\nonumber\\
 &=\;2^{m} \frac{2 \pi^{m}}{\Gamma\left(m\right)}\int_{M}\bigg\{-3f^{-4m+2}T(u,v,w)+mf^{-4m+1}\bigg(g(u,w)v(f)-g(v,w)u(f)-g(u,v)w(f)\bigg)\bigg\}d{\rm Vol}_M.
 \end{align}
 \end{thm}

\begin{proof}
Substituting the symbols of $f(D+c(T))f$ into (\ref{ABD}), then by (\ref{abd}), we need to compute the following three parts {\bf (I)-(III)}.

\noindent {\bf (I)} For $c(u)c(v)c(w)\sigma_{-2 m}\left((f(D+c(T))f)^{-2 m}\right)\sigma_{0}(f(D+c(T))f)(x_{0})$:
\begin{align}\label{0-2m}
&c(u)c(v)c(w)\sigma_{-2 m}\left((f(D+c(T))f)^{-2 m}\right)\sigma_{0}(f(D+c(T))f)(x_{0})\nonumber\\
&=f^{-4m}\|\xi\|^{-2m}c(u)c(v)c(w)\left(-\frac{1}{4}f\sum_{ist}w_{st}(e_i)c(e_i)c(e_s)c(e_t)f\right)(x_0)\nonumber\\
&+f^{-4m}\|\xi\|^{-2m}c(u)c(v)c(w)\left(fc(T)f\right)(x_0)\nonumber\\
&=f^{-4m+2}\|\xi\|^{-2m}c(u)c(v)c(w)c(T),
\end{align}
then by lemma \ref{pppp} and further integration calculation, we get
\begin{align}\label{m8}
	&\int_{\|\xi\|=1}\operatorname{tr}\biggl\{ c(u)c(v)c(w)\sigma_{-2 m}\left((f(D+c(T))f)^{-2 m}\right)\sigma_{0}(f(D+c(T))f) \biggr\}(x_0)\sigma(\xi)\nonumber\\
&=\int_{\|\xi\|=1}\frac{3}{2}f^{-4m+2}\|\xi\|^{-2m}T(u,v,w){\rm tr}[id]\sigma(\xi)\nonumber\\
&=\frac{3}{2}f^{-4m+2}T(u,v,w){\rm tr}[id]{\rm Vol}(S^{n-1}).
\end{align}

\noindent{\bf (II)}
For $c(u)c(v)c(w)\sigma_{-2 m-1}\left((f(D+c(T))f)^{-2 m}\right)\sigma_{1}(f(D+c(T))f)(x_0)$:
\begin{align}\label{2-2m-2}
	&c(u)c(v)c(w)\sigma_{-2 m-1}\left((f(D+c(T))f)^{-2 m}\right)\sigma_{1}(f(D+c(T))f)(x_0)\nonumber\\
	&=mf^{-4m+2}\|\xi\|^{-2m-2}\sum_{\mu}\xi_\mu c(u)c(v)c(w)c(\partial^\mu)c(T)c(\xi)(x_0)\nonumber\\
&+mf^{-4m+2}\|\xi\|^{-2m-2}\sum_{\mu}\xi_\mu c(u)c(v)c(w)c(T)c(\partial^\mu)c(\xi)(x_0)\nonumber\\
	&-mf^{-4m-1}\|\xi\|^{-2m} c(u)c(v)c(w)c(df^3)(x_0)\nonumber\\
&+mf^{-4m+1}\|\xi\|^{-2m-2} c(u)c(v)c(w)c(\xi)c(df)c(\xi)(x_0)\nonumber\\
&-2mf^{-4m+6}\|\xi\|^{-2m-2}\sum_{\mu}\xi_\mu\partial_{x_\mu}(f^{-4})c(u)c(v)c(w)c(\xi)(x_0)\nonumber\\
&+f^{-4m+6}\sum_{k=0}^{m-2}\sum_{\mu}(-m+k+1)\|\xi\|^{-2m}\xi_\mu\partial_{x_\mu}(f^{-4})c(u)c(v)c(w)c(\xi)(x_0).
\end{align}
The same as the calculation process of (\ref{m8}), the next step is to perform trace and integral operations on the above six interms into (\ref{2-2m-2}).

\noindent{\bf (II-a)}
\begin{align}\label{t7}
	&\int_{\|\xi\|=1}mf^{-4m+2}\|\xi\|^{-2m-2}\sum_{\mu}\xi_\mu \operatorname{tr}\bigg(c(u)c(v)c(w)c(\partial^\mu)c(T)c(\xi)\bigg)\sigma(\xi)\nonumber\\
&=\int_{\|\xi\|=1}\frac{3}{2}mf^{-4m+2}\|\xi\|^{-2m-2}\sum_{1\leq j<l<t\leq n}\sum_{\mu}\xi_\mu \operatorname{tr}\biggl\{c(u)c(v)c(w) T(e_j,e_l,e_t)c(\partial^\mu)c(e_j)c(e_l)c(e_t)c(\xi)\biggr\}\sigma(\xi)\nonumber\\
&=\int_{\|\xi\|=1}\frac{3}{2}mf^{-4m+2}\|\xi\|^{-2m-2}\sum_{\substack{1\leq j<l<t\leq n\\\mu,s,r,p,q}}T(e_j,e_l,e_t)u_rv_pw_q\xi_\mu\xi_s{\rm tr}\bigg(c(e_r)c(e_p)c(e_q)c(e_\mu)c(e_j)c(e_l)c(e_t)c(e_s)\bigg)\sigma(\xi)\nonumber\\
&=\frac{3}{2}mf^{-4m+2}\delta_\mu^s\times\frac{1}{2m}{\rm Vol}(S^{n-1})\sum_{\substack{1\leq j<l<t\leq n\\\mu,s,r,p,q}}T(e_j,e_l,e_t)u_rv_pw_q{\rm tr}\bigg(c(e_r)c(e_p)c(e_q)c(e_\mu)c(e_j)c(e_l)c(e_t)c(e_s)\bigg)\nonumber\\
&=\frac{3}{4}f^{-4m+2}{\rm Vol}(S^{n-1})\sum_{\substack{1\leq j<l<t\leq n\\s,r,p,q}}T(e_j,e_l,e_t)u_rv_pw_q{\rm tr}\bigg(c(e_r)c(e_p)c(e_q)c(e_s)c(e_j)c(e_l)c(e_t)c(e_s)\bigg),
\end{align}
where
\begin{align}\label{qq}
&{\rm tr}\bigg(c(e_r)c(e_p)c(e_q)c(e_s)c(e_j)c(e_l)c(e_t)c(e_s)\bigg)\nonumber\\
&=-2{\rm tr}\bigg(c(e_r)c(e_p)c(e_q)c(e_l)c(e_t)c(e_j)\bigg)+2\delta^l_s{\rm tr}\bigg(c(e_r)c(e_p)c(e_q)c(e_j)c(e_t)c(e_s)\bigg)\nonumber\\
&-2\delta^t_s{\rm tr}\bigg(c(e_r)c(e_p)c(e_q)c(e_j)c(e_l)c(e_s)\bigg)-{\rm tr}\bigg(c(e_r)c(e_p)c(e_q)c(e_j)c(e_l)c(e_t)c(e_s)c(e_s)\bigg)\nonumber\\
&=(2m-6){\rm tr}\bigg(c(e_r)c(e_p)c(e_q)c(e_j)c(e_l)c(e_t)\bigg).
\end{align}
Substituting (\ref{qq}) into (\ref{t7}), we have
\begin{align}\label{t17}
	&\int_{\|\xi\|=1}mf^{-4m+2}\|\xi\|^{-2m-2}\sum_{\mu}\xi_\mu \operatorname{tr}\bigg(c(u)c(v)c(w)c(\partial^\mu)c(T)c(\xi)\bigg)\nonumber\\
&=(m-3)f^{-4m+2}{\rm Vol}(S^{n-1}){\rm tr}\bigg(c(u)c(v)c(w)c(T)\bigg)\nonumber\\
&=\frac{3}{2}(m-3)f^{-4m+2}T(u,v,w){\rm tr}[id]{\rm Vol}(S^{n-1}).
\end{align}
\noindent{\bf (II-b)}Similarly, we obtain
\begin{align}\label{tpp7}
&\int_{\|\xi\|=1}mf^{-4m+2}\|\xi\|^{-2m-2}\sum_{\mu}\xi_\mu \operatorname{tr}\bigg(c(u)c(v)c(w)c(T)c(\partial^\mu)c(\xi)\bigg)\nonumber\\
&=\int_{\|\xi\|=1} \frac{3}{2}mf^{-4m+2}\|\xi\|^{-2m-2}\sum_{\substack{1\leq j<l<t\leq n\\\mu,s,r,p,q}}T(e_j,e_l,e_t)u_rv_pw_q\xi_\mu\xi_s{\rm tr}\bigg(c(e_r)c(e_p)c(e_q)c(e_j)c(e_l)c(e_t)c(e_\mu)c(e_s)\bigg)\sigma(\xi)\nonumber\\
&=\frac{3}{2}mf^{-4m+2}\times\frac{1}{2m}{\rm Vol}(S^{n-1})\sum_{\substack{1\leq j<l<t\leq n\\s,r,p,q}}T(e_j,e_l,e_t)u_rv_pw_q{\rm tr}\bigg(c(e_r)c(e_p)c(e_q)c(e_j)c(e_l)c(e_t)c(e_s)c(e_s)\bigg)\nonumber\\
&=-\frac{3}{2}mf^{-4m+2}\times T(u,v,w){\rm tr}[id]{\rm Vol}(S^{n-1}).
\end{align}

\noindent{\bf (II-c)}Let $g(u,v)=\sum_{i=1}^nu_iv_i$ and $w(f)=\sum_{i=1}^n\partial_{x_i}(f)w_i$, we get
\begin{align}
	&\int_{\|\xi\|=1}-mf^{-4m-1}\|\xi\|^{-2m} \operatorname{tr} \biggl\{c(u)c(v)c(w)c(df^3)\biggr\}(x_0)\sigma(\xi)\nonumber\\
	&=mf^{-4m-1}\bigg(g(u,w)v(f^3)-g(v,w)u(f^3)-g(u,v)w(f^3)\bigg)
{\rm tr}[id]{\rm Vol}(S^{n-1})\nonumber\\
&=3mf^{-4m+1}\bigg(g(u,w)v(f)-g(v,w)u(f)-g(u,v)w(f)\bigg)
{\rm tr}[id]{\rm Vol}(S^{n-1}).
\end{align}

\noindent{\bf (II-d)}By the relation of the Clifford action and $\operatorname{tr}(AB) =\operatorname{tr}(BA)$, we have the equality:
\begin{align}\label{444}
	&mf^{-4m+1}\|\xi\|^{-2m-2}\operatorname{tr} \biggl\{ c(u)c(v)c(w)c(\xi)c(df)c(\xi)\biggr\}(x_0)\nonumber\\
&=mf^{-4m+1}\|\xi\|^{-2m-2}\sum_{s,\mu,t,r,p,q}\xi_s\xi_t\partial_{x_\mu}(f){\rm tr}\bigg(c(e_r)c(e_p)c(e_q)c(e_s)\c(e_\mu) c(e_t)\bigg)\nonumber\\
&=mf^{-4m+1}\|\xi\|^{-2m-2}\sum_{s,\mu,t,r,p,q}\xi_s\xi_t\partial_{x_\mu}(f)\bigg(\delta_r^p(-\delta_q^s\delta_\mu^t+\delta_q^\mu\delta_s^t-\delta_q^t\delta_s^\mu)+\delta_r^q(\delta_p^s\delta_\mu^t+\delta_p^\mu\delta_s^t-\delta_p^t\delta_s^\mu)\nonumber\\
&+\delta_r^s(-\delta_q^p\delta_\mu^t+\delta_p^\mu\delta_q^t-\delta_p^t\delta_q^\mu)+\delta_r^\mu(\delta_q^p\delta_s^t+\delta_q^t\delta_s^p-\delta_p^t\delta_s^q)+\delta_r^t(-\delta_q^p\delta_\mu^s+\delta_q^\mu\delta_s^p-\delta_p^\mu\delta_s^q)\bigg){\rm tr}[id]\nonumber\\
&=mf^{-4m+1}\|\xi\|^{-2m-2}\sum_{s,\mu,t,r,p,q}\bigg(-\xi_s\xi_t\partial_{x_t}(f)u_pv_pw_s+\xi_s\xi_s\partial_{x_\mu}(f)u_pv_pw_\mu-\xi_s\xi_t\partial_{x_t}(f)u_pv_pw_t\nonumber\\
&+xi_s\xi_t\partial_{x_t}(f)u_qv_sw_q-\xi_s\xi_s\partial_{x_p}(f)u_qv_pw_q+\xi_s\xi_t\partial_{x_s}(f)u_qv_tw_q-\xi_s\xi_t\partial_{x_t}(f)u_sv_pw_p+\xi_s\xi_t\partial_{x_p}(f)u_sv_pw_t\nonumber\\
&-\xi_s\xi_t\partial_{x_q}(f)u_sv_tw_p+\xi_s\xi_s\partial_{x_r}(f)u_rv_pw_p-\xi_s\xi_t\partial_{x_r}(f)u_rv_sw_t+\xi_s\xi_t\partial_{x_r}(f)u_rv_tw_s-\xi_s\xi_t\partial_{x_s}(f)u_tv_pw_p\nonumber\\
&+\xi_s\xi_t\partial_{x_q}(f)u_tv_sw_q-\xi_s\xi_t\partial_{x_p}(f)u_tv_pw_s\bigg){\rm tr}[id].
\end{align}
Integrating the result in (\ref{444}), we get
\begin{align}
	&\int_{\|\xi\|=1}mf^{-4m+1}\|\xi\|^{-2m-2}\operatorname{tr} \biggl\{ c(u)c(v)c(w)c(\xi)c(df)c(\xi)\biggr\}(x_0)\sigma(\xi)\nonumber\\
&=(1-m)f^{-4m+1}\bigg(g(u,w)v(f)-g(u,v)w(f)-g(v,w)u(f)\bigg){\rm tr}[id]{\rm Vol}(S^{n-1}).
\end{align}

\noindent{\bf (II-e)}
\begin{align}
	&\operatorname{tr} \biggl\{-2mf^{-4m+6}\|\xi\|^{-2m-2}\xi^\mu\partial_{x_\mu}(f^{-4})c(u)c(v)c(w)c(\xi)\biggr\}(x_0)\nonumber\\
	&=-2mf^{-4m+6}\|\xi\|^{-2m-2}\sum_{\mu,r,p,q}\xi_\mu\partial_{x_\mu}(f^{-4})\bigg(\xi_qu_rv_rw_q-\xi_pu_rv_pw_r+\xi_ru_rv_pw_p\bigg){\rm tr}[id].
\end{align}
By direct computations, we have
\begin{align}
	&\int_{\|\xi\|=1}-2mf^{-4m+6}\|\xi\|^{-2m-2}\xi_\mu\partial_{x_\mu}(f^{-4})\operatorname{tr} \biggl\{c(u)c(v)c(w)c(\xi)\biggr\}(x_0)\sigma(\xi)\nonumber\\
	&=f^{-4m+6}\bigg(g(u,w)v(f^{-4})-g(u,v)w(f^{-4})-g(v,w)u(f^{-4})\bigg){\rm tr}[id]{\rm Vol}(S^{n-1})\nonumber\\
&=-4f^{-4m+1}\bigg(g(u,w)v(f)-g(u,v)w(f)-g(v,w)u(f)\bigg){\rm tr}[id]{\rm Vol}(S^{n-1}).
\end{align}

\noindent{\bf (II-f)}
\begin{align}
	&\operatorname{tr} \biggl\{\sum_{k=0}^{m-2}\sum_{\mu}f^{-4m+6}(-m+k+1)\|\xi\|^{-2m}\xi^\mu\partial_{x_\mu}(f^{-4})c(u)c(v)c(w)c(\xi)\biggr\}(x_0)\nonumber\\
	&=\sum_{k=0}^{m-2}\sum_{\mu,r,p,q}f^{-4m+6}(-m+k+1)\|\xi\|^{-2m}\xi^\mu\partial_{x_\mu}(f^{-4})\bigg(\xi_qu_rv_rw_q-\xi_pu_rv_pw_r+\xi_ru_rv_pw_p\bigg){\rm tr}[id].
\end{align}
Also, straight forward computations yield
\begin{align}
&\int_{\|\xi\|=1}\operatorname{tr} \biggl\{\sum_{k=0}^{m-2}\sum_{\mu}f^{-4m+6}(-m+k+1)\|\xi\|^{-2m}\xi_\mu\partial_{x_\mu}(f^{-4})c(u)c(v)c(w)c(\xi)\biggr\}(x_0)\sigma(\xi)\nonumber\\
&=(m-1)f^{-4m+1}\bigg(g(u,w)v(f)-g(u,v)w(f)-g(v,w)u(f)\bigg){\rm tr}[id]{\rm Vol}(S^{n-1}).
\end{align}

Finally, we get
\begin{align}
	&\int_{\|\xi\|=1} \operatorname{tr}\bigg(c(u)c(v)c(w)\sigma_{-2 m-1}\left((f(D+c(T))f)^{-2 m}\right)\sigma_{1}(f(D+c(T))f)(x_0)\bigg) \sigma(\xi)\nonumber\\
	&=\bigg\{-\frac{9}{2}f^{-4m+2}T(u,v,w)+(m-2)f^{-4m+1}\bigg(g(u,w)v(f)+g(v,w)u(f)-g(u,v)w(f)\bigg)\bigg\}{\rm tr}[id]{\rm Vol}(S^{n-1}).
\end{align}

\noindent{\bf (III)} For $-\sqrt{-1} c(u)c(v)c(w)\sum_{j=1}^{2m} \partial_{\xi_{j}}\left[\sigma_{-2 m}\left((f(D+c(T))f)^{-2 m}\right)\right] \partial_{x_{j}}\left[\sigma_{1}(f(D+c(T))f)\right]$:

\begin{align}\label{2-2m-1}
	&-\sqrt{-1} c(u)c(v)c(w) \sum_{j=1}^{2m} \partial_{\xi_{j}}\left[\sigma_{-2 m}\left((f(D+c(T))f)^{-2 m}\right)\right] \partial_{x_{j}}\left[\sigma_{1}(f(D+c(T))f)\right]\nonumber\\
	&=-4mf^{-4m+1}\|\xi\|^{-2m-2}\partial_{x_\alpha}(f)\xi^\alpha c(u)c(v)c(w)c(\xi)(x_0)\nonumber\\
&-2mf^{-4m+2}\|\xi\|^{-2m-2}\xi^\alpha  c(u)c(v)c(w)\partial_{x_\alpha}[c(\xi)](x_0)\nonumber\\
&=-4mf^{-4m+1}\|\xi\|^{-2m-2}\partial_{x_\alpha}(f)\xi_\alpha c(u)c(v)c(w)c(\xi).
\end{align}
By the relation of the Clifford action and $\operatorname{tr}(AB) =\operatorname{tr}(BA)$, we have the equality:
\begin{align}
	&\operatorname{tr} \biggl\{-\sqrt{-1} c(u)c(v)c(w) \sum_{j=1}^{2m} \partial_{\xi_{j}}\left[\sigma_{-2 m}\left((f(D+c(T))f)^{-2 m}\right)\right] \partial_{x_{j}}\left[\sigma_{1}(f(D+c(T))f)\right]\biggr\}(x_0)\nonumber\\
	&=-4mf^{-4m+1}\|\xi\|^{-2m-2}\partial_{x_\alpha}(f)\xi_\alpha \operatorname{tr}\biggl(c(u)c(v)c(w)c(\xi)\biggr)(x_0)\nonumber\\
&=-4mf^{-4m+1}\|\xi\|^{-2m-2}\partial_{x_\alpha}(f)\bigg(\xi_\alpha\xi_qu_rv_rw_q-\xi_\alpha\xi_pu_rv_pw_r+\xi_\alpha\xi_ru_rv_pw_p\bigg){\rm tr}[id]
\end{align}
Moreover, in the same way, we have
\begin{align}
	&\int_{\|\xi\|=1}\operatorname{tr} \biggl\{-\sqrt{-1} c(u)c(v)c(w) \sum_{j=1}^{2m} \partial_{\xi_{j}}\left[\sigma_{-2 m}\left((f(D+c(T))f)^{-2 m}\right)\right] \partial_{x_{j}}\left[\sigma_{1}(f(D+c(T))f)\right]\biggr\}(x_0)\sigma(\xi)\nonumber\\
	&=-2f^{-4m+1}\bigg(\partial_{x_\alpha}(f)u_rv_rw_\alpha-\partial_{x_\alpha}(f)u_rv_\alpha w_r+\partial_{x_\alpha}(f)u_\alpha v_p w_p\bigg){\rm tr}[id]{\rm Vol}(S^{n-1})\nonumber\\
&=2f^{-4m+1}\bigg(g(u,w)v(f)-g(u,v)w(f)-g(v,w)u(f)\bigg){\rm tr}[id]{\rm Vol}(S^{n-1}).
\end{align}
Finally, summing up the results in {\bf (I)}-{\bf (III)} to get
\begin{align}\label{z2}
&	\mathscr{S}_{f(D+c(T))f}\bigg(c(u),c(v),c(w)\bigg)\nonumber\\
	&=2^{m} \frac{2 \pi^{m}}{\Gamma\left(m\right)}\int_{M}\bigg\{-3f^{-4m+2}T(u,v,w)+mf^{-4m+1}\bigg(g(u,w)v(f)-g(v,w)u(f)-g(u,v)w(f)\bigg)\bigg\}d{\rm Vol}_M.
\end{align}
Hence, Theorem \ref{thm} holds.

\end{proof}
\subsection{The spectral torsion for the rescaled Dirac operator $f(D+\sqrt{-1}c(X))f$}
Let $\mathbb{A}=\sqrt{-1}c(X),$  then the rescaled Dirac operator $f(D+\sqrt{-1}c(X))f$ is defined as:
\begin{align}\label{mrp1}
f(D+\sqrt{-1}c(X))f	&=f\bigg[\sum^n_{i=1}c(e_i)\bigg(e_i-\frac{1}{4}\sum_{s,t}\omega_{s,t}
(e_i)c(e_s)c(e_t)\bigg)+\sqrt{-1}c(X)\bigg]f,
\end{align}
where $c(X)=\sum_{\alpha=1}^nX_\alpha c(e_\alpha).$

By the relation of the Clifford action and ${\rm tr}(AB) ={\rm tr}(BA)$, we have the following lemma.
\begin{lem}\label{mrpppp}
The following identity holds:
\begin{align}
&{\rm tr}\bigg(c(u)c(v)c(w)c(X)\bigg)=\bigg(g(u,v)g(w,X)-g(u,w)g(v,X)+g(v,w)g(u,X)\bigg)
{\rm tr}[id].
\end{align}
\end{lem}
\begin{proof}
\begin{align}\label{mrt12}
&{\rm tr}\bigg(c(u)c(v)c(w)c(X)\bigg)\nonumber\\
&=\sum_{\alpha,r,p,q}{\rm tr}\bigg(c(e_r)c(e_p)c(e_q)c(e_\alpha)\bigg)\nonumber\\
&=\sum_{\alpha,r,p,q}\bigg(X_\alpha u_rv_pw_q(\delta^r_p\delta_q^\alpha-\delta_r^q\delta_\alpha^p+\delta_r^\alpha\delta_p^q)\bigg){\rm tr}[id]\nonumber\\
&=\sum_{\alpha,r,p,q}\bigg(X_qu_rv_rw_q-X_pu_qv_pw_q+X_ru_rv_pw_p\bigg){\rm tr}[id]\nonumber\\
&=\bigg(g(u,v)g(w,X)-g(u,w)g(v,X)+g(v,w)g(u,X)\bigg)
{\rm tr}[id].
\end{align}
\end{proof}
Now we arrive to our second main result.
 \begin{thm}\label{mrthm}
 	Let $M$ be an $n=2m$ dimensional ($n\geq 3$) oriented compact spin Riemannian manifold, for the rescaled Dirac operator with the trilinear Clifford multiplication by functional of differential one-forms $c(u),c(v),c(w),$ the spectral torsion for $f(D+\sqrt{-1}c(X))f$ equals to
 \begin{align}
 	&\mathscr{S}_{f(D+\sqrt{-1}c(X))f}\bigg(c(u),c(v),c(w)\bigg)\nonumber\\
 &=\;2^{m} \frac{2 \pi^{m}}{\Gamma\left(m\right)}\int_{M}mf^{-4m+1}\bigg(g(u,w)v(f)-g(v,w)u(f)-g(u,v)w(f)\bigg)d{\rm Vol}_M.
 \end{align}
 \end{thm}

\begin{proof}
Similar to Section 4.1, we need to calculate the following three parts.

For $c(u)c(v)c(w)\sigma_{-2 m}\left((f(D+\sqrt{-1}c(X))f)^{-2 m}\right)\sigma_{0}(f(D+\sqrt{-1}c(X))f)(x_{0})$:
\begin{align}\label{mr0-2m}
&c(u)c(v)c(w)\sigma_{-2 m}\left((f(D+\sqrt{-1}c(X))f)^{-2 m}\right)\sigma_{0}(f(D+\sqrt{-1}c(X))f)(x_{0})\nonumber\\
&=f^{-4m}\|\xi\|^{-2m}c(u)c(v)c(w)\left(-\frac{1}{4}f\sum_{ist}w_{st}(e_i)c(e_i)c(e_s)c(e_t)f\right)(x_0)\nonumber\\
&+\sqrt{-1}f^{-4m}\|\xi\|^{-2m}c(u)c(v)c(w)\left(fc(X)f\right)(x_0)\nonumber\\
&=\sqrt{-1}f^{-4m+2}\|\xi\|^{-2m}c(u)c(v)c(w)c(X).
\end{align}
then by lemma \ref{mrpppp}, we get
\begin{align}\label{mrm8}
	&\int_{\|\xi\|=1}\operatorname{tr}\biggl\{ c(u)c(v)c(w)\sigma_{-2 m}\left((f(D+\sqrt{-1}c(X))f)^{-2 m}\right)\sigma_{0}(f(D+\sqrt{-1}c(X))f) \biggr\}(x_0)\sigma(\xi)\nonumber\\
&=\int_{\|\xi\|=1}\sqrt{-1}f^{-4m+2}\|\xi\|^{-2m}\operatorname{tr}\biggl\{ c(u)c(v)c(w)c(X)\biggr\}\sigma(\xi)\nonumber\\
&=\sqrt{-1}f^{-4m+2}\bigg(g(u,v)g(w,X)-g(u,w)g(v,X)+g(v,w)g(u,X)\bigg){\rm tr}[id]{\rm Vol}(S^{n-1}).
\end{align}

For $c(u)c(v)c(w)\sigma_{-2 m-1}\left((f(D+\sqrt{-1}c(X))f)^{-2 m}\right)\sigma_{1}(f(D+\sqrt{-1}c(X))f)(x_0)$:
\begin{align}\label{mr2-2m-2}
	&c(u)c(v)c(w)\sigma_{-2 m-1}\left((f(D+\sqrt{-1}c(X))f)^{-2 m}\right)\sigma_{1}(f(D+\sqrt{-1}c(X))f)(x_0)\nonumber\\
	&=\sqrt{-1}mf^{-4m+2}\|\xi\|^{-2m-2}\sum_{\mu}\xi_\mu c(u)c(v)c(w)c(\partial^\mu)c(X)c(\xi)(x_0)\nonumber\\
&+\sqrt{-1}mf^{-4m+2}\|\xi\|^{-2m-2}\sum_{\mu}\xi_\mu c(u)c(v)c(w)c(X)c(\partial^\mu)c(\xi)(x_0)\nonumber\\
	&-mf^{-4m-1}\|\xi\|^{-2m} c(u)c(v)c(w)c(df^3)(x_0)\nonumber\\
&+mf^{-4m+1}\|\xi\|^{-2m-2} c(u)c(v)c(w)c(\xi)c(df)c(\xi)(x_0)\nonumber\\
&-2mf^{-4m+6}\|\xi\|^{-2m-2}\sum_{\mu}\xi_\mu\partial_{x_\mu}(f^{-4})c(u)c(v)c(w)c(\xi)(x_0)\nonumber\\
&+f^{-4m+6}\sum_{k=0}^{m-2}\sum_{\mu}(-m+k+1)\|\xi\|^{-2m}\xi_\mu\partial_{x_\mu}(f^{-4})c(u)c(v)c(w)c(\xi)(x_0).
\end{align}
The results of the first two items are as follows, and the results of the last four items are the same as those in Section 4.1.

\noindent{\bf (1)}

\begin{align}\label{mrt7}
	&\int_{\|\xi\|=1}mf^{-4m+2}\|\xi\|^{-2m-2}\sum_{\mu}\xi_\mu \operatorname{tr}\bigg(c(u)c(v)c(w)c(\partial^\mu)c(X)c(\xi)\bigg)\sigma(\xi)\nonumber\\
&=\int_{\|\xi\|=1}mf^{-4m+2}\|\xi\|^{-2m-2}X_\alpha\operatorname{tr}\biggl\{c(u)c(v)c(w)\sum_{\mu,\alpha}\xi_\mu c(\partial^\mu)c(e_\alpha)c(\xi)\biggr\}\sigma(\xi)\nonumber\\
&=\int_{\|\xi\|=1}mf^{-4m+2}\|\xi\|^{-2m-2}\sum_{\mu,\alpha,r,p,q,s}\xi_\mu \xi_su_rv_pw_qX_\alpha\operatorname{tr}\biggl\{c(e_r)c(e_p)c(e_q) c(\partial^\mu)c(e_\alpha)c(e_s)\biggr\}\sigma(\xi)\nonumber\\
&=mf^{-4m+2}\|\xi\|^{-2m-2}\times\frac{1}{2m}{\rm Vol}(S^{n-1})\sum_{\mu,\alpha,r,p,q,s}\delta^\mu_su_rv_pw_qX_\alpha\operatorname{tr}\biggl\{c(e_r)c(e_p)c(e_q) c(e_\mu)c(e_\alpha)c(e_s)\biggr\}\nonumber\\
&=\frac{1}{2}f^{-4m+2}\|\xi\|^{-2m-2}{\rm Vol}(S^{n-1})\sum_{\alpha,r,p,q,s}u_rv_pw_qX_\alpha\operatorname{tr}\biggl\{c(e_r)c(e_p)c(e_q) c(e_s)c(e_\alpha)c(e_s)\biggr\}\nonumber\\
\end{align}
where
\begin{align}\label{mrqq}
&\operatorname{tr}\biggl\{c(e_r)c(e_p)c(e_q) c(e_s)c(e_\alpha)c(e_s)\biggr\}\nonumber\\
&=-2\delta^\alpha_s{\rm tr}\bigg(c(e_r)c(e_p)c(e_q)c(e_s)\bigg)-{\rm tr}\bigg(c(e_r)c(e_p)c(e_q)c(e_\alpha)c(e_s)c(e_s)\bigg)\nonumber\\
&=2(m-1){\rm tr}\bigg(c(e_r)c(e_p)c(e_q)c(e_\alpha)\bigg).
\end{align}
Substituting (\ref{mrqq}) into (\ref{mrt7}), we have
\begin{align}\label{mrt17}
	&\int_{\|\xi\|=1}\sqrt{-1}mf^{-4m+2}\|\xi\|^{-2m-2}\sum_{\mu}\xi_\mu \operatorname{tr}\bigg(c(u)c(v)c(w)c(\partial^\mu)c(X)c(\xi)\bigg)\nonumber\\
&=\frac{\sqrt{-1}}{2}f^{-4m+2}{\rm Vol}(S^{n-1}){\rm tr}\bigg(c(u)c(v)c(w)c(X)\bigg)\nonumber\\
&=\sqrt{-1}(m-1)\bigg(g(u,v)g(w,X)-g(u,w)g(v,X)+g(v,w)g(u,X)\bigg){\rm tr}[id]{\rm Vol}(S^{n-1}).
\end{align}
\noindent{\bf (2)}Similarly, we obtain
\begin{align}\label{mrtpp7}
&\int_{\|\xi\|=1}\sqrt{-1}mf^{-4m+2}\|\xi\|^{-2m-2}\sum_{\mu}\xi_\mu \operatorname{tr}\bigg(c(u)c(v)c(w)c(X)c(\partial^\mu)c(\xi)\bigg)\sigma(\xi)\nonumber\\
&=\int_{\|\xi\|=1}\sqrt{-1}mf^{-4m+2}\|\xi\|^{-2m-2}X_\alpha\operatorname{tr}\biggl\{c(u)c(v)c(w)\sum_{\mu,\alpha}\xi_\mu c(e_\alpha)c(\partial^\mu)c(\xi)\biggr\}\sigma(\xi)\nonumber\\
&=\int_{\|\xi\|=1}\sqrt{-1}mf^{-4m+2}\|\xi\|^{-2m-2}\sum_{\mu,\alpha,r,p,q,s}\xi_\mu \xi_su_rv_pw_qX_\alpha\operatorname{tr}\biggl\{c(e_r)c(e_p)c(e_q) c(e_\alpha)c(\partial^\mu)c(e_s)\biggr\}\sigma(\xi)\nonumber\\
&=\sqrt{-1}mf^{-4m+2}\|\xi\|^{-2m-2}\times\frac{1}{2m}{\rm Vol}(S^{n-1})\sum_{\mu,\alpha,r,p,q,s}\delta^\mu_su_rv_pw_qX_\alpha\operatorname{tr}\biggl\{c(e_r)c(e_p)c(e_q) c(e_\alpha)c(e_\mu)c(e_s)\biggr\}\nonumber\\
&=\frac{\sqrt{-1}}{2}f^{-4m+2}\|\xi\|^{-2m-2}{\rm Vol}(S^{n-1})\sum_{\alpha,r,p,q,s}u_rv_pw_qX_\alpha\operatorname{tr}\biggl\{c(e_r)c(e_p)c(e_q) c(e_\alpha)c(e_s)c(e_s)\biggr\}\nonumber\\
&=-\sqrt{-1}mf^{-4m+2}{\rm Vol}(S^{n-1}){\rm tr}\bigg(c(u)c(v)c(w)c(X)\bigg)\nonumber\\
&=-\sqrt{-1}mf^{-4m+2}\bigg(g(u,v)g(w,X)-g(u,w)g(v,X)+g(v,w)g(u,X)\bigg){\rm tr}[id]{\rm Vol}(S^{n-1}).
\end{align}
Finally, we get
\begin{align}
	&\int_{\|\xi\|=1} \operatorname{tr}\bigg(c(u)c(v)c(w)\sigma_{-2 m-1}\left((f(D+\sqrt{-1}c(X))f)^{-2 m}\right)\sigma_{1}(f(D+\sqrt{-1}c(X))f)(x_0)\bigg) \sigma(\xi)\nonumber\\
	&=(m-2)f^{-4m+1}\bigg(g(u,w)v(f)+g(v,w)u(f)-g(u,v)w(f)\bigg){\rm tr}[id]{\rm Vol}(S^{n-1}).
\end{align}
The results of third iterm in (\ref{ABD}) is the same as in section 4.1. Thus, summing up the above results, we get
\begin{align}\label{mrz2}
&	\mathscr{S}_{f(D+\sqrt{-1}c(X))f}\bigg(c(u),c(v),c(w)\bigg)\nonumber\\
	&=2^{m} \frac{2 \pi^{m}}{\Gamma\left(m\right)}\int_{M}mf^{-4m+1}\bigg(g(u,w)v(f)-g(v,w)u(f)-g(u,v)w(f)\bigg)d{\rm Vol}_M.
\end{align}
Hence, Theorem \ref{mrthm} holds.
\end{proof}

\subsection{The spectral torsion for the rescaled Dirac operator $f(D+c(X)\gamma )f$}

 Let us begin by a technical lemma showing that the product of the grading $\gamma$ by any
 Euclidean Dirac matrix results. The grading operator $\gamma$ denoted by

 \begin{align}
 \gamma=(\sqrt{-1})^m\prod_{j=1}^{2m}c(e_j).
 \end{align}
 In the terms of the orthonormal frames ${e_i}(1\leq i,j\leq n)$ on $TM$, we have $\gamma=(\sqrt{-1})^mc(e_1)c(e_2)\cdot\cdot\cdot c(e_n)$.
\begin{defn}\cite{YY}
 Suppose $V$ is a super vector space, and $\gamma$ is its super structure. If
 $\phi \in End(V)$, let $tr (\phi)$ be the trace of $\phi$, then define
 \begin{align}
 {\rm Str}(\phi)={\rm tr}(\phi\circ\gamma),
 \end{align}
where ${\rm tr}$ and $ {\rm Str}$ are called the trace and the super trace of $\phi$ respectively.
 \end{defn}
 \begin{lem}\cite{YY}
 The super trace (function) $Str:End_\mathbb{C}(S(2m))\rightarrow \mathbb{C}$ is a complex linear
 map satisfying
 \begin{align}
{\rm Str}(c(e_{i1})c(e_{i2})\cdot\cdot\cdot c(e_{iq}))=\left\{
                                                     \begin{array}{ll}
                                                       0, & if q< 2m; \\
                                                      \frac{2^m}{(\sqrt{-1})^m}, & if q= 2m ,
                                                     \end{array}
                                                   \right.
 \end{align}
where $1\leq i_1,i_2\cdot\cdot\cdot i_q\leq2m.$
 \end{lem}

Let $\mathbb{A}=c(X)\gamma,$  then the rescaled Dirac operator $f(D+c(X)\gamma)f$ is defined as:
\begin{align}\label{rp1}
f(D+c(X)\gamma)f	&=f\bigg[\sum^n_{i=1}c(e_i)\bigg(e_i-\frac{1}{4}\sum_{s,t}\omega_{s,t}
(e_i)c(e_s)c(e_t)\bigg)+c(X)\gamma\bigg]f.
\end{align}
\begin{lem}\cite{WWj}\label{rpppp}
The following identity holds:
\begin{align}
&{\rm tr}\bigg(c(u)c(v)c(w)c(X)\gamma\bigg)=\left\{
                                              \begin{array}{ll}
                                                0, & if~~~~ 2m\neq 4; \\
                                                 -4\langle u^*\wedge v^*\wedge w^*\wedge X^*,e_1^*\wedge e_2^*\wedge e_3^*\wedge e_4^*\rangle , & if~~~~ 2m=4.
                                              \end{array}
                                            \right.
\end{align}
\end{lem}
Now we arrive to our third main result.
 \begin{thm}\label{rthm}
 	Let $M$ be an $n=2m$ dimensional ($n\geq 3$) oriented compact spin Riemannian manifold, for the rescaled Dirac operator with the trilinear Clifford multiplication by functional of differential one-forms $c(u),c(v),c(w),$ the spectral torsion for $f(D+c(X)\gamma)f$ equals to  the following identities.

When $2m\neq 4,$
 \begin{align}
 	&\mathscr{S}_{f(D+c(X)\gamma)f}\bigg(c(u),c(v),c(w)\bigg)\nonumber\\
 &=\;2^{m} \frac{2 \pi^{m}}{\Gamma\left(m\right)}\int_{M}mf^{-4m+1}\bigg(g(u,w)v(f^3)-g(v,w)u(f^3)-g(u,v)w(f^3)\bigg)d{\rm Vol}_M.
 \end{align}
When $2m=4,$
\begin{align}
 	&\mathscr{S}_{f(D+c(X)\gamma)f}\bigg(c(u),c(v),c(w)\bigg)\nonumber\\
 &=\;16 \pi^{2}\int_{M}f^{-6}u^*\wedge v^*\wedge w^*\wedge X^*+8 \pi^{2}\int_{M}2f^{-7}\bigg(g(u,w)v(f^3)-g(v,w)u(f^3)-g(u,v)w(f^3)\bigg)d{\rm Vol}_M.
 \end{align}
 \end{thm}
\begin{proof}
 For $c(u)c(v)c(w)\sigma_{-2 m}\left((f(D+c(X)\gamma)f)^{-2 m}\right)\sigma_{0}(f(D+c(X)\gamma)f)(x_{0})$:
\begin{align}\label{r0-2m}
&c(u)c(v)c(w)\sigma_{-2 m}\left((f(D+c(X)\gamma)f)^{-2 m}\right)\sigma_{0}(f(D+c(X)\gamma)f)(x_{0})\nonumber\\
&=f^{-4m}\|\xi\|^{-2m}c(u)c(v)c(w)\left(-\frac{1}{4}f\sum_{ist}w_{st}(e_i)c(e_i)c(e_s)c(e_t)f\right)(x_0)\nonumber\\
&+f^{-4m}\|\xi\|^{-2m}c(u)c(v)c(w)\left(fc(X)\gamma f\right)(x_0)\nonumber\\
&=f^{-4m+2}\|\xi\|^{-2m}c(u)c(v)c(w)c(X)\gamma.
\end{align}
then by lemma \ref{rpppp}, we get
\begin{align}\label{rm8}
	&\int_{\|\xi\|=1}\operatorname{tr}\biggl\{ c(u)c(v)c(w)\sigma_{-2 m}\left((f(D+c(X)\gamma)f)^{-2 m}\right)\sigma_{0}(f(D+c(X)\gamma)f) \biggr\}(x_0)\sigma(\xi)\nonumber\\
&=\int_{\|\xi\|=1} f^{-4m+2}\|\xi\|^{-2m}\operatorname{tr}\biggl\{ c(u)c(v)c(w)c(X)\gamma\biggr\}\sigma(\xi)\nonumber\\
&=\left\{
                                              \begin{array}{ll}
                                                0, & if~~~~2m\neq 4; \\
                                                 -4f^{-6} u^*\wedge v^*\wedge w^*\wedge X^* {\rm Vol}(S^{3}), & if ~~~~2m=4.
                                              \end{array}
                                            \right.
\end{align}
For $c(u)c(v)c(w)\sigma_{-2 m-1}\left((f(D+c(X)\gamma)f)^{-2 m}\right)\sigma_{1}(f(D+c(X)\gamma)f)(x_0)$:
\begin{align}\label{r2-2m-2}
	&c(u)c(v)c(w)\sigma_{-2 m-1}\left((f(D+c(X)\gamma)f)^{-2 m}\right)\sigma_{1}(f(D+c(X)\gamma)f)(x_0)\nonumber\\
	&=mf^{-4m+2}\|\xi\|^{-2m-2}\sum_{\mu}\xi_\mu c(u)c(v)c(w)c(\partial^\mu)c(X)\gamma c(\xi)(x_0)\nonumber\\
&+mf^{-4m+2}\|\xi\|^{-2m-2}\sum_{\mu}\xi_\mu c(u)c(v)c(w)c(X)\gamma c(\partial^\mu)c(\xi)(x_0)\nonumber\\
	&-mf^{-4m-1}\|\xi\|^{-2m} c(u)c(v)c(w)c(df^3)(x_0)\nonumber\\
&+mf^{-4m+1}\|\xi\|^{-2m-2} c(u)c(v)c(w)c(\xi)c(df)c(\xi)(x_0)\nonumber\\
&-2mf^{-4m+6}\|\xi\|^{-2m-2}\sum_{\mu}\xi_\mu\partial_{x_\mu}(f^{-4})c(u)c(v)c(w)c(\xi)(x_0)\nonumber\\
&+f^{-4m+6}\sum_{k=0}^{m-3}\sum_{\mu}(-m+k+1)\|\xi\|^{-2m}\xi_\mu\partial_{x_\mu}(f^{-4})c(u)c(v)c(w)c(\xi)(x_0).
\end{align}
The results of the first two items are as follows, and the results of the last four items are the same as those
in Section 4.1.

\noindent{\bf (1)}
\begin{align}\label{rt7}
	&\int_{\|\xi\|=1}mf^{-4m+2}\|\xi\|^{-2m-2}\sum_{\mu}\xi_\mu \operatorname{tr}\bigg(c(u)c(v)c(w)c(\partial^\mu)c(X)\gamma c(\xi)\bigg)\sigma(\xi)\nonumber\\
&=\int_{\|\xi\|=1}mf^{-4m+2}\|\xi\|^{-2m-2}\sum_{\mu,\alpha,r,p,q,s}\xi_\mu \xi_su_rv_pw_qX_\alpha\operatorname{tr}\biggl\{c(e_r)c(e_p)c(e_q) c(e_\mu)c(e_\alpha)\gamma c(e_s)\biggr\}\sigma(\xi)\nonumber\\
&=-\int_{\|\xi\|=1}mf^{-4m+2}\|\xi\|^{-2m-2}\sum_{\mu,\alpha,r,p,q,s}\xi_\mu \xi_su_rv_pw_qX_\alpha\operatorname{tr}\biggl\{c(e_r)c(e_p)c(e_q) c(e_\mu)c(e_\alpha)c(e_s)\gamma \biggr\}\sigma(\xi)\nonumber\\
&=-mf^{-4m+2}\|\xi\|^{-2m-2}\times\frac{1}{2m}{\rm Vol}(S^{n-1})\sum_{\mu,\alpha,r,p,q,s}\delta^\mu_su_rv_pw_qX_\alpha\operatorname{tr}\biggl\{c(e_r)c(e_p)c(e_q) c(e_\mu)c(e_\alpha)c(e_s)\gamma\biggr\}\nonumber\\
&=-mf^{-4m+2}\|\xi\|^{-2m-2}\times\frac{1}{2m}{\rm Vol}(S^{n-1})\sum_{\alpha,r,p,q,s}u_rv_pw_qX_\alpha\operatorname{tr}\biggl\{c(e_r)c(e_p)c(e_q) c(e_s)c(e_\alpha)c(e_s)\gamma\biggr\},
\end{align}
where
\begin{align}\label{rqq}
&\operatorname{tr}\biggl\{c(e_r)c(e_p)c(e_q) c(e_s)c(e_\alpha)c(e_s)\gamma\biggr\}\nonumber\\
&=-2\delta^\alpha_s{\rm tr}\bigg(c(e_r)c(e_p)c(e_q)c(e_s)\gamma\bigg)-{\rm tr}\bigg(c(e_r)c(e_p)c(e_q)c(e_\alpha)c(e_s)c(e_s)\gamma\bigg)\nonumber\\
&=2(m-1){\rm tr}\bigg(c(e_r)c(e_p)c(e_q)c(e_\alpha)\gamma\bigg).
\end{align}
Substituting (\ref{rqq}) into (\ref{rt7}), we have
\begin{align}\label{rt17}
	&\int_{\|\xi\|=1}mf^{-4m+2}\|\xi\|^{-2m-2}\sum_{\mu}\xi_\mu \operatorname{tr}\bigg(c(u)c(v)c(w)c(\partial^\mu)c(X)\gamma c(\xi)\bigg)\sigma(\xi)\nonumber\\
&=-mf^{-4m+2}\frac{1}{2m}{\rm Vol}(S^{n-1})2(m-1){\rm tr}\bigg(c(u)c(v)c(w)c(X)\bigg)\nonumber\\
&=-(m-1){\rm tr}\bigg(c(u)c(v)c(w)c(X)\bigg){\rm Vol}(S^{n-1})\nonumber\\
&=\left\{
    \begin{array}{ll}
      0, & if~~~~ 2m\neq 4; \\
     4f^{-6}u^*\wedge v^*\wedge w^*\wedge X^*{\rm Vol}(S^{3}), & if~~~~ 2m=4.
    \end{array}
  \right.
\end{align}
\noindent{\bf (2)}Similarly, we obtain
\begin{align}\label{rtpp7}
&\int_{\|\xi\|=1}mf^{-4m+2}\|\xi\|^{-2m-2}\sum_{\mu}\xi_\mu \operatorname{tr}\bigg(c(u)c(v)c(w)c(X)\gamma c(\partial^\mu)c(\xi)\bigg)\sigma(\xi)\nonumber\\
&=\int_{\|\xi\|=1}mf^{-4m+2}\|\xi\|^{-2m-2}\sum_{\mu,\alpha,r,p,q,s}\xi_\mu \xi_su_rv_pw_qX_\alpha\operatorname{tr}\biggl\{c(e_r)c(e_p)c(e_q) c(e_\alpha)\gamma c(\partial^\mu)c(e_s)\biggr\}\sigma(\xi)\nonumber\\
&=mf^{-4m+2}\|\xi\|^{-2m-2}\times\frac{1}{2m}{\rm Vol}(S^{n-1})\sum_{\mu,\alpha,r,p,q,s}\delta^\mu_su_rv_pw_qX_\alpha\operatorname{tr}\biggl\{c(e_r)c(e_p)c(e_q) c(e_\alpha)\gamma c(e_\mu)c(e_s)\biggr\}\nonumber\\
&=mf^{-4m+2}\|\xi\|^{-2m-2}\times\frac{1}{2m}{\rm Vol}(S^{n-1})\sum_{\alpha,r,p,q,s}u_rv_pw_qX_\alpha\operatorname{tr}\biggl\{c(e_r)c(e_p)c(e_q) c(e_\alpha)\gamma c(e_s)c(e_s)\biggr\}\nonumber\\
&=-mf^{-4m+2}{\rm Vol}(S^{n-1}){\rm tr}\bigg(c(u)c(v)c(w)c(X)\gamma\bigg)\nonumber\\
&=\left\{
    \begin{array}{ll}
      0, & if~~~~ 2m\neq 4; \\
     8f^{-6}u^*\wedge v^*\wedge w^*\wedge X^*{\rm Vol}(S^{3}), & if~~~~ 2m=4.
    \end{array}
  \right..
\end{align}

The results of anther four items in (\ref{r2-2m-2}) are the same as in Section 4.1. The results of third iterm in (\ref{ABD}) is the same as in section 4.1. Thus, when $2m\neq 4,$ we have
 \begin{align}
 	&\mathscr{S}_{f(D+c(X)\gamma)f}\bigg(c(u),c(v),c(w)\bigg)\nonumber\\
 &=\;2^{m} \frac{2 \pi^{m}}{\Gamma\left(m\right)}\int_{M}mf^{-4m+1}\bigg(g(u,w)v(f)-g(v,w)u(f)-g(u,v)w(f)\bigg)d{\rm Vol}_M.
 \end{align}
When $2m=4,$ we have
\begin{align}
 	&\mathscr{S}_{f(D+c(X)\gamma)f}\bigg(c(u),c(v),c(w)\bigg)\nonumber\\
 &=\;16 \pi^{2}\int_{M}f^{-6}u^*\wedge v^*\wedge w^*\wedge X^*+8 \pi^{2}\int_{M}2f^{-7}\bigg(g(u,w)v(f)-g(v,w)u(f)-g(u,v)w(f)\bigg)d{\rm Vol}_M.
 \end{align}
Hence, Theorem \ref{rthm} holds.
\end{proof}

\section{The spectral torsion for the rescaled Dirac operator $f(D^E+\gamma \otimes \Phi)f$}
\label{section:5}
In this section, we consider an additional smooth vector bundle $E$ over $M$ (with $C^\infty(M)$-module
 of smooth sections $W$), equipped with a connection $\nabla^E$, with corresponding curvature-tensor
 $R^E$. And the tensor product vector bundle $S(TM)\otimes E$ is equiped with the compound
 connection:
 $$\nabla^{S(TM)\otimes E}=\nabla^{S(TM)}\otimes {\rm id_E}+{\rm id_{S(TM)}}\otimes\nabla^E,$$

where $\nabla^{S(TM)}$ is a spin conneection on the spinor bundle, defined by $\nabla^{S(TM)}_X=X+\frac{1}{4}\sum_{s,t}\langle \nabla^L_Xe_s,e_t\rangle c(e_s)c(e_t)$. The corresponding twisted
 Dirac operator $D^E$ is locally specified as follows:
\begin{align}
D^E=\sum_{i,j}g^{ij}c(\partial_i)\nabla^{S(TM)\otimes E}_{\partial_j}=\sum_{i=1}^nc(e_i)\nabla^{S(TM)\otimes E}_{e_i},
\end{align}
where $\nabla^{S(TM)\otimes E}_{\partial_j}=\partial_j+\sigma_j^s+\sigma_j^E$ and $\sigma_j^s=\frac{1}{4}\sum_{s,t}\langle \nabla^L_Xe_s,e_t\rangle c(e_s)c(e_t)$, $\sigma_j^E$ is a spin connection matrix of $E$.

Then the rescaled Dirac operator $f(D^E+\gamma \otimes \Phi)f$ is defined as:
\begin{align}\label{brp1}
f(D^E+\gamma \otimes \Phi)f	&=f\bigg(\sum_{i=1}^nc(e_i)\nabla^{S(TM)\otimes E}_{e_i}+\gamma \otimes \Phi\bigg)f,
\end{align}
where $\Phi\in\Gamma(M,End(E))$ and $\Phi=\Phi^*.$

From (6a) in \cite{Ka}, we have
\begin{align}
D_E^2&=-g^{ij}\partial_i\partial_j-2\sigma^j_{S(TM)\otimes E}\partial_j+\Gamma^k\partial_k-g^{ij}[\partial_i(\sigma^j_{S(TM)\otimes E})+\sigma^i_{S(TM)\otimes E}\sigma^j_{S(TM)\otimes E}-\Gamma^k_{ij}\sigma^k_{S(TM)\otimes E}]\nonumber\\
&+\frac{1}{4}s+\frac{1}{2}\sum_{i\neq j}R^E(e_i,e_j)c(e_i)c(e_j).
\end{align}

Moreover,
\begin{align}
(D^E+\gamma \otimes \Phi)^2&=-g^{ij}\partial_i\partial_j+\Gamma^k\partial_k+g^{ij}[-2\sigma_i^{S(TM)\otimes E}+c(\partial_i)(\gamma \otimes \Phi)+\gamma \otimes \Phi c(\partial_i)]\partial_j\nonumber\\
&+g^{ij}[-\partial_i(\sigma_j^{S(TM)\otimes E})-\sigma_i^{S(TM)\otimes E}\sigma_j^{S(TM)\otimes E}+(\gamma \otimes \Phi )c(\partial_i)\partial_j+\gamma \otimes \Phi c(\partial_i)\sigma_j^{S(TM)\otimes E}\nonumber\\
&+\Gamma^k_{ij}\sigma_k^{S(TM)\otimes E}+c(\partial_i)\partial_j(\gamma \otimes \Phi )+c(\partial_i)\sigma_j^{S(TM)\otimes E}(\gamma \otimes \Phi )]+\gamma^2 \otimes \Phi^2\nonumber\\
&+\frac{1}{4}s+\frac{1}{2}\sum_{i\neq j}R^E(e_i,e_j)c(e_i)c(e_j).
\end{align}
Similar to (\ref{p4}), we expand $(f(D^E+\gamma \otimes \Phi)f)^2$,
\begin{align}\label{pyyy4}
(f(D^E+\gamma \otimes \Phi)f)^2&=f(D^E+\gamma \otimes \Phi)f^2(D^E+\gamma \otimes \Phi)f\nonumber\\
&=f^4(D^E+\gamma \otimes \Phi)^2+fc(df^3)(D^E+\gamma \otimes \Phi)+f^3(D^E+\gamma \otimes \Phi)c(df)+fc(df^2)c(df).
\end{align}
Then we get the following lemma.
\begin{lem}
The leading symbols of $f(D^E+\gamma \otimes \Phi)f$ and $(f(D^E+\gamma \otimes \Phi)f)^2$:
\begin{align}
&\sigma_{2}[(f(D^E+\gamma \otimes \Phi)f)^2](x,\xi)= f^4\|\xi\|^2;\nonumber\\
		&\sigma_{1}[(f(D^E+\gamma \otimes \Phi)f)^2](x,\xi)=\sqrt{-1}f^4\bigg(\Gamma^j-2\sigma^j+c(\partial^j)\triangle+\triangle c(\partial^j)\bigg)\xi_j+\sqrt{-1}fc(df^3)c(\xi)+\sqrt{-1}f^3c(\xi)c(df);\nonumber\\
&\sigma_{1}[f(D^E+\gamma \otimes \Phi)f](x,\xi)= \sqrt{-1}f^2c(\xi);\nonumber\\
		&\sigma_{0}[f(D^E+\gamma \otimes \Phi)f](x,\xi)=-\frac{1}{4}f\sum_{i,s,t}\omega_{st}(e_i)c(e_i)c(e_s)c(e_t)f+f\sum_{j=1}^nc(e_j)\sigma_j^Ef+f\gamma \otimes \Phi f.
\end{align}
\end{lem}
Now we are in position to prove the main result in this section.
 \begin{thm}\label{brthm}
 	Let $M$ be an $n=2m$ dimensional ($n\geq 3$) oriented compact spin Riemannian manifold, for the rescaled Dirac operator with the trilinear Clifford multiplication by functional of differential one-forms $c(u),c(v),c(w),$ the spectral torsion for $f(D^E+\gamma \otimes \Phi)f$ equals to
 \begin{align}
 	&\mathscr{S}_{f(D^E+\gamma \otimes \Phi)f}\bigg(c(u),c(v),c(w)\bigg)\nonumber\\
 &=\;2^{m} \frac{2 \pi^{m}}{\Gamma\left(m\right)}\int_{M}mf^{-4m+1}\bigg(g(u,w)v(f)-g(v,w)u(f)-g(u,v)w(f)\bigg)d{\rm Vol}_M.
 \end{align}
 \end{thm}
\begin{proof}In the normal coordinates $U$ of $x_0$ in $M$, $\sigma_j^E(x_0)=0.$

 For $c(u)c(v)c(w)\sigma_{-2 m}\left((f(D^E+\gamma \otimes \Phi)f)^{-2 m}\right)\sigma_{0}(f(D^E+\gamma \otimes \Phi)f)(x_{0})$:
\begin{align}\label{br0-2m}
&c(u)c(v)c(w)\sigma_{-2 m}\left((f(D^E+\gamma \otimes \Phi)f)^{-2 m}\right)\sigma_{0}(f(D^E+\gamma \otimes \Phi)f)(x_{0})\nonumber\\
&=f^{-4m}\|\xi\|^{-2m}c(u)c(v)c(w)\left(-\frac{1}{4}f\sum_{ist}w_{st}(e_i)c(e_i)c(e_s)c(e_t)f\right)(x_0)\nonumber\\
&+f^{-4m+2}\|\xi\|^{-2m}c(u)c(v)c(w)\sum_{j=1}^{n}c(e_j)\sigma_j^E(x_0)\nonumber\\
&+f^{-4m+2}\|\xi\|^{-2m}c(u)c(v)c(w)\gamma\otimes\Phi(x_0)\nonumber\\
&=f^{-4m+2}\|\xi\|^{-2m}c(u)c(v)c(w)\gamma\otimes\Phi.
\end{align}
 We note $$\operatorname{tr}\biggl( c(u)c(v)c(w)\gamma\biggr)=\operatorname{Str}\biggl( c(u)c(v)c(w)\biggr)=0,$$ and $$\operatorname{tr}\biggl( c(u)c(v)c(w)\gamma\otimes\Phi\biggr)=\operatorname{tr}\biggl( c(u)c(v)c(w)\gamma\biggr)\operatorname{tr}\biggl( c(u)c(v)c(w)\Phi\biggr),$$ then further integration, we get
\begin{align}\label{brm8}
	&\int_{\|\xi\|=1}\operatorname{tr}\biggl\{ c(u)c(v)c(w)\sigma_{-2 m}\left((f(D^E+\gamma \otimes \Phi)f)^{-2 m}\right)\sigma_{0}(f(D^E+\gamma \otimes \Phi)f) \biggr\}(x_0)\sigma(\xi)\nonumber\\
&=\int_{\|\xi\|=1} f^{-4m+2}\|\xi\|^{-2m}\operatorname{tr}\biggl( c(u)c(v)c(w)\gamma\biggr)\operatorname{tr}\biggl( c(u)c(v)c(w)\Phi\biggr)\sigma(\xi)\nonumber\\
&=0.
\end{align}
For $c(u)c(v)c(w)\sigma_{-2 m-1}\left((f(D^E+\gamma \otimes \Phi)f)^{-2 m}\right)\sigma_{1}(f(D^E+\gamma \otimes \Phi)f)(x_0)$:
\begin{align}\label{br2-2m-2}
	&c(u)c(v)c(w)\sigma_{-2 m-1}\left((f(D^E+\gamma \otimes \Phi)f)^{-2 m}\right)\sigma_{1}(f(D^E+\gamma \otimes \Phi)f)(x_0)\nonumber\\
	&=mf^{-4m+2}\|\xi\|^{-2m-2}\sum_{\mu}\xi_\mu c(u)c(v)c(w)c(\partial^\mu)\gamma \otimes \Phi c(\xi)(x_0)\nonumber\\
&+mf^{-4m+2}\|\xi\|^{-2m-2}\sum_{\mu}\xi_\mu c(u)c(v)c(w)\gamma \otimes \Phi c(\partial^\mu)c(\xi)(x_0)\nonumber\\
	&-mf^{-4m-1}\|\xi\|^{-2m} c(u)c(v)c(w)c(df^3)(x_0)\nonumber\\
&+mf^{-4m+1}\|\xi\|^{-2m-2} c(u)c(v)c(w)c(\xi)c(df)c(\xi)(x_0)\nonumber\\
&-2mf^{-4m+6}\|\xi\|^{-2m-2}\sum_{\mu}\xi_\mu\partial_{x_\mu}(f^{-4})c(u)c(v)c(w)c(\xi)(x_0)\nonumber\\
&+f^{-4m+6}\sum_{k=0}^{m-3}\sum_{\mu}(-m+k+1)\|\xi\|^{-2m}\xi_\mu\partial_{x_\mu}(f^{-4})c(u)c(v)c(w)c(\xi)(x_0).
\end{align}
\noindent{\bf (1)}By the relation of the Clifford action and $\operatorname{tr}(AB)=\operatorname{tr}(BA)$, we have the equality:

\begin{align}\label{brt7}
	&\int_{\|\xi\|=1}mf^{-4m+2}\|\xi\|^{-2m-2}\sum_{\mu}\xi_\mu \operatorname{tr}\bigg(c(u)c(v)c(w)c(\partial^\mu)\gamma c(\xi)\bigg)\sigma(\xi)\nonumber\\
&=\int_{\|\xi\|=1}mf^{-4m+2}\|\xi\|^{-2m-2}\sum_{\mu,\alpha,r,p,q,s}\xi_\mu \xi_su_rv_pw_qX_\alpha\operatorname{tr}\biggl\{c(e_r)c(e_p)c(e_q) c(e_\mu)\gamma c(e_s)\biggr\}\sigma(\xi)\nonumber\\
&=-\int_{\|\xi\|=1}mf^{-4m+2}\|\xi\|^{-2m-2}\sum_{\mu,\alpha,r,p,q,s}\xi_\mu \xi_su_rv_pw_qX_\alpha\operatorname{tr}\biggl\{c(e_r)c(e_p)c(e_q) c(e_\mu)c(e_s)\gamma \biggr\}\sigma(\xi)\nonumber\\
&=-mf^{-4m+2}\times\frac{1}{2m}{\rm Vol}(S^{n-1})\sum_{\mu,\alpha,r,p,q,s}\delta^\mu_su_rv_pw_qX_\alpha\operatorname{tr}\biggl\{c(e_r)c(e_p)c(e_q) c(e_\mu)c(e_s)\gamma\biggr\}\nonumber\\
&=-mf^{-4m+2}\times\frac{1}{2m}{\rm Vol}(S^{n-1})\sum_{\alpha,r,p,q,s}u_rv_pw_qX_\alpha\operatorname{tr}\biggl\{c(e_r)c(e_p)c(e_q) c(e_s)c(e_s)\gamma\biggr\}\nonumber\\
&=mf^{-4m+2}\times{\rm Vol}(S^{n-1})\sum_{\alpha,r,p,q,s}u_rv_pw_qX_\alpha\operatorname{tr}\biggl\{c(e_r)c(e_p)c(e_q) \gamma\biggr\}\nonumber\\
&=mf^{-4m+2}{\rm tr}\bigg(c(u)c(v)c(w)\gamma\bigg){\rm Vol}(S^{n-1})\nonumber\\
&=0.
\end{align}
\noindent{\bf (2)}Similarly, we obtain
\begin{align}\label{brtpp7}
&\int_{\|\xi\|=1}mf^{-4m+2}\|\xi\|^{-2m-2}\sum_{\mu}\xi_\mu \operatorname{tr}\bigg(c(u)c(v)c(w)\gamma c(\partial^\mu)c(\xi)\bigg)\sigma(\xi)\nonumber\\
&=\int_{\|\xi\|=1}mf^{-4m+2}\|\xi\|^{-2m-2}\sum_{\mu,\alpha,r,p,q,s}\xi_\mu \xi_su_rv_pw_qX_\alpha\operatorname{tr}\biggl\{c(e_r)c(e_p)c(e_q) c(e_\alpha)\gamma c(\partial^\mu)c(e_s)\biggr\}\sigma(\xi)\nonumber\\
&=mf^{-4m+2}\|\xi\|^{-2m-2}\times\frac{1}{2m}{\rm Vol}(S^{n-1})\sum_{\mu,\alpha,r,p,q,s}\delta^\mu_su_rv_pw_qX_\alpha\operatorname{tr}\biggl\{c(e_r)c(e_p)c(e_q) c(e_\alpha)\gamma c(e_\mu)c(e_s)\biggr\}\nonumber\\
&=mf^{-4m+2}\|\xi\|^{-2m-2}\times\frac{1}{2m}{\rm Vol}(S^{n-1})\sum_{\alpha,r,p,q,s}u_rv_pw_qX_\alpha\operatorname{tr}\biggl\{c(e_r)c(e_p)c(e_q) c(e_\alpha)\gamma c(e_s)c(e_s)\biggr\}\nonumber\\
&=-mf^{-4m+2}{\rm Vol}(S^{n-1}){\rm tr}\bigg(c(u)c(v)c(w)\gamma\bigg)\nonumber\\
&=0.
\end{align}

The results of anther four items in (\ref{brm8}) are the same as in Section 4.1. The results of third iterm in (\ref{ABD}) is the same as in section 4.1. Thus, we have
\begin{align}
	&\int_{\|\xi\|=1}\operatorname{tr} \biggl\{-\sqrt{-1} c(u)c(v)c(w) \sum_{j=1}^{2m} \partial_{\xi_{j}}\left[\sigma_{-2 m}\left((f(D^E+\gamma \otimes \Phi)f)^{-2 m}\right)\right] \partial_{x_{j}}\left[\sigma_{0}(f(D^E+\gamma \otimes \Phi)f)\right]\biggr\}(x_0)\sigma(\xi)\nonumber\\
&=2f^{-4m+1}\bigg(g(u,w)v(f)-g(u,v)w(f)-g(v,w)u(f)\bigg){\rm tr}[id]{\rm Vol}(S^{n-1}).
\end{align}
Thus, summing up the results in {\bf (I)}-{\bf (III)} to get
\begin{align}\label{brz2}
&	\mathscr{S}_{f(D^E+\gamma \otimes \Phi)f}\bigg(c(u),c(v),c(w)\bigg)\nonumber\\
	&=2^{m} \frac{2 \pi^{m}}{\Gamma\left(m\right)}\int_{M}mf^{-4m+1}\bigg(g(u,w)v(f)-g(v,w)u(f)-g(u,v)w(f)\bigg)d{\rm Vol}_M.
\end{align}
Hence, Theorem \ref{brthm} holds.
\end{proof}

\section*{ Declarations}
\textbf{Ethics approval and consent to participate:} Not applicable.

\textbf{Consent for publication:} Not applicable.

\textbf{Availability of data and materials:} The authors confrm that the data supporting the findings of this study are available within the article.

\textbf{Competing interests:} The authors declare no competing interests.

\textbf{Author Contributions:} All authors contributed to the study conception and design. Material preparation,
data collection and analysis were performed by TW and YW. The first draft of the manuscript was written
by TW and all authors commented on previous versions of the manuscript. All authors read and approved
the final manuscript.

\section*{Acknowledgements}
This first author was supported by NSFC. No.12401059 and Liaoning Province Science and Technology Plan Joint Project 2023-BSBA-118. The second author was supported NSFC. No.11771070. The authors thank the referee for his (or her) careful reading and helpful comments.

\end{document}